\documentclass[11pt,dvips,twoside]{article}

\usepackage{pslatex}
\usepackage{fancyhdr}
\usepackage{graphicx}
\usepackage{geometry}

\RequirePackage{amsfonts,amssymb,amsmath,amscd,amsthm}
\RequirePackage{txfonts}
\RequirePackage{graphicx}
\RequirePackage{xcolor}
\RequirePackage{geometry}
\RequirePackage{enumerate}

\def\figurename{Figure} % Replace the colon that normally appears after the Figure number by a period.
\makeatletter
\renewcommand{\fnum@figure}[1]{\figurename~\thefigure.}
\makeatother

\def\tablename{Table} % Replace the colon that normally appears after the Figure number by a period.
\makeatletter
\renewcommand{\fnum@table}[1]{\tablename~\thetable.}
\makeatother

\usepackage{amsmath}
\usepackage{amssymb}
\usepackage{amsfonts}
\usepackage{amsthm,amscd}

\newtheorem{theorem}{Theorem}[section]

\newtheorem{proposition}[theorem]{proposition}

\theoremstyle{example}
\newtheorem{example}[theorem]{example}

\theoremstyle{definition}
\newtheorem{definition}[theorem]{definition}

\theoremstyle{remark}
\newtheorem{remark}[theorem]{Remark}

\numberwithin{equation}{section}

\begin{document}

\title{\bfseries\scshape{Some structures of Hom-Poisson color algebras}}

\author{\bfseries\scshape Ibrahima BAKAYOKO\thanks{e-mail address: ibrahimabakayoko27@gmail.com}\\
D\'epartement de Math\'ematiques,
Universit\'e de N'Z\'er\'ekor\'e,\\
BP 50 N'Z\'er\'ekor\'e, Guin\'ee.
  \\\bfseries\scshape Silvain Attan \thanks{e-mail address: syltane2010@yahoo.fr}\\
 D\'epartement de Math\'ematiques,
Universit\'e d'Abomey Calavi,\\ 01 BP 4521, Cotonou 01, B\'enin.
}
 
\date{}
\maketitle 

% % % % % % \thispagestyle{empty} \setcounter{page}{1}
% % % % % % % ------- [First Page Running Head] - place it immediately after title! ------
% % % % % % \thispagestyle{fancy} \fancyhead{}
% % % % % % \fancyhead[L]{{\LARGE A}frican {\LARGE D}iaspora {\LARGE J}ournal of {\LARGE M}athematics\\
% % % % % % Volume X, Number X, pp. {\thepage--\pageref{lastpage-01} (2013)}} % put \label{lastpage-xx} on the last page!
% % % % % % \fancyhead[R]{ISSN 1539-854X  \\ {\tt{www.math-res-pub.org/adjm}}}
% % % % % % \fancyfoot{}
% % % % % % \renewcommand{\headrulewidth}{0pt}
%------------------------------------------------------------------------------

\noindent\hrulefill

\noindent {\bf Abstract.} 
In many previous papers, the authors used an endomorphism of algebra to twist the original algebraic structures in order to produce the corresponding
Hom-algebraic structures. In this works, we use these either a bijective linear map, either an element of centroid either an averaging operator 
either Nijenhuis operator, either a multiplier to produce Hom-Poisson color algebras from given one.

\noindent \hrulefill

\vspace{.3in}


\vspace{.08in} \noindent \textbf{Keywords}: 
Hom-Poisson color algebras, bijective even linear map, element of centroid, averaging operator, Nijenhuis operator,  Rota-Baxter operator
 and multiplier.
\vspace{.3in}
% \noindent {\small Revised: April 16, 2014, June 28, 2014 $\parallel$  Accepted: July 4, 2014}
\vspace{.2in}
\section{Introduction and first definitions}
 Hom-associative color algebras \cite{LY} has been introduced by L. Yuan  as a generalization of both Hom-associative algebra and associative color algebras.
And Hom-Lie color algebras were introduced by the same auther as natural a generalization of the Hom-Lie algebras as well as a special case
of the quasi-hom-Lie algebras. The author proved that the commutator of any Hom-associative color algebra gives rise to Hom-Lie color algebra and
 presents a way to obtain Hom-Lie color algebras from the classical Lie color algebras along with algebra  endomorphisms. Also, He introduced a 
multiplier $\sigma$ on an abelian group  and provide constructions of new  Hom-Lie color algebras from old ones by the $\sigma$-twists. 

However, Hom-Poisson color algebras are introduced in \cite{BI1} as the colored version verion of Hom-Poisson algebras introduced in \cite{IT}. 
The authers in \cite{IT} give some constructions of Hom-Poisson color algebras from  Hom-associative color algebras which twisting map is an
averaging operator or from a given Hom-Poisson color algebra together with an averaging operator or from a Hom-post-Poisson color algebra. 
In particular, they show that any Hom-pre-Poisson color algebra leads to a Hom-Poisson color algebra. 
The description of Hom-Poisson color algebras  is given in \cite{SA} by using only one operation of its two binary operations via the 
polarisation-depolarisation process.

The goal of this paper is to give a contuation of constructions of Hom-Poisson color algebras \cite{IT}.
While many authers working on Hom-algebras use a morphism of Hom-algebras to builds another one, we ask our self if there are other kind of twist 
which are not morphism such that we can get Hom-algebraic structures from others one. 
To give a positive answer to the questions above, we organize this paper as follows. 
In Section 1, we recall some basic definitions about Rota-Baxter Hom-associative color algebras and Rota-Baxter Hom-Lie color algebras
 as well as averaging operator, Nijenhuis operator and centroid.
In Section 2, we give the main results i.e. we give new products for Hom-Poisson color algebras from another one by twisting the original 
multiplications by a bijective linear map, an element of centroid, an averaging operator, a Rota-Baxter operator, a Nijenhuis operator or a
 multiplier.

% ------------ [Running Heads - for odd and even pages] - please insert it only on page 2!
\pagestyle{fancy} \fancyhead{} \fancyhead[EC]{ } 
\fancyhead[EL,OR]{\thepage} \fancyhead[OC]{Ibrahima bakayoko and Silvain Attan} \fancyfoot{}
\renewcommand\headrulewidth{0.5pt}
%------------------------------------------------------------------------------

 Throughout this paper, all graded vector spaces are assumed to be over a field $\mathbb{K}$ of characteristic different from 2.
% \section{Preliminaries}
% In this section, we recall definition of Rota-Baxter Hom-associative and Rota-Baxter Hom-Lie color algebras 
% as well as averaging operator, Nijenhuis operator and centroid.

\begin{definition}
 Let $G$ be an abelian group. A map $\varepsilon :G\times G\rightarrow {\bf \mathbb{K}^*}$ is called a skew-symmetric bicharacter on $G$ if the following
identities hold, 
\begin{enumerate} 
 \item [(i)] $\varepsilon(a, b)\varepsilon(b, a)=1$,
\item [(ii)] $\varepsilon(a, b+c)=\varepsilon(a, b)\varepsilon(a, c)$,
\item [(iii)]$\varepsilon(a+b, c)=\varepsilon(a, c)\varepsilon(b, c)$,
\end{enumerate}
$a, b, c\in G$.
\end{definition}
\begin{remark}
Observe that $\varepsilon(a, 0)=\varepsilon(0, a)=1, \varepsilon(a,a)=\pm 1 \;\mbox{for all}\; a\in G, \;\mbox{where}\; 0 \;\mbox{is the identity of}\; G.$
\end{remark}
If x and y are two homogeneous elements of degree $a$ and $b$ respectively and $\varepsilon$ is a skew-symmetric bicharacter, 
then we shorten the notation by writing $\varepsilon(x, y)$ instead of $\varepsilon(a, b)$.

\begin{definition}
 By a color Hom-algebra we mean a quadruple $(A, \mu, \varepsilon, \alpha)$ in which 
\begin{enumerate}
 \item [a)] $A$ is a $G$-graded vector space i.e. $A=\bigoplus_{a\in G}A_a$,
\item [b)] $\mu : A \times A \rightarrow A$ is an even bilinear map i.e. $\mu(A_a, A_b)\subseteq A_{a+b}$, for all $a, b\in G$,
\item [c)] $\alpha : A\rightarrow A$ is an even linear map i.e. $\alpha(A_a)\subseteq A_{a}$,
\item [d)] $\varepsilon : G\times G\rightarrow{\bf K}^*$ is a bicharacter.
\end{enumerate}
\end{definition}

\begin{example}
%  Some standard examples of skew-symmetric bicharacters are:
 \begin{enumerate}
\item [1)] 
$ G=\mathbb{Z}_2^n=\{(\alpha_1, \dots, \alpha_n)| \alpha_i\in\mathbb{Z}_2 \}, \quad
\varepsilon((\alpha_1, \dots, \alpha_n), (\beta_1, \dots, \beta_n)):= (-1)^{\alpha_1\beta_1+\dots+\alpha_n\beta_n},$
\item [2)] $G=\mathbb{Z}\times\mathbb{Z} ,\quad \varepsilon((i_1, i_2), (j_1, j_2))=(-1)^{(i_1+i_2)(j_1+j_2)}$,
\item [3)] $G=\{-1, +1\} , \quad\varepsilon(i, j)=(-1)^{(i-1)(j-1)/{4}}$.
\end{enumerate}
\end{example}

\begin{example}
 Let $\sigma : G\times G\rightarrow\mathbb{K}^*$ be any mapping such that 
\begin{eqnarray}
 \sigma(x, y+z)\sigma(y, z)=\sigma(x, y)\sigma(x+y, z), \forall x, y, z\in G.
\end{eqnarray}
Then, $\delta(x, y)=\sigma(x, y)\sigma(y, x)^{-1}$ is a bicharacter on $G$. In this case, $\sigma$ is called a {\it multiplier} on $G$, and 
$\delta$ the bicharacter associated with $\sigma$.

For instance, let us define the mapping $\sigma : G\times G\rightarrow\mathbb{R}$ by
\begin{eqnarray}
 \sigma((i_1, i_2), (j_1, j_2))=(-1)^{i_1j_2}, \forall i_k, j_k\in\mathbb{Z}_2, k=1, 2.\nonumber
\end{eqnarray}
It is easy to verify that $\sigma$ is a multiplier on $G$ and 
\begin{eqnarray}
 \delta((i_1, i_2), (j_1, j_2))=(-1)^{i_1j_2-i_2j_1}, \forall i_k, j_k\in \mathbb{Z}_2, i=1, 2. \nonumber
\end{eqnarray}
is a bicharacter on $G$.
\end{example}

 \begin{definition}
A Hom-associative color algebra is a color Hom-algebra $(A, \mu, \varepsilon, \alpha)$  such that
\begin{eqnarray}
 as_\mu(x, y, z)=\mu(\alpha(x), \mu(y, z))-\mu(\mu(x, y), \alpha(z))=0, \label{aca}
\end{eqnarray}
for any $x, y, z \in \mathcal{H}(A)$.\\
If in addition $\mu=\varepsilon(\cdot, \cdot)\mu^{op}$ i.e. $\mu(x, y)=\varepsilon(x, y)\mu(y, x)$, for any $x, y\in\mathcal{H}(A)$, the Hom-associative color algebra 
$(A, \mu, \varepsilon, \alpha)$ is said to be a commutative Hom-associative color algebra.
\end{definition}
\begin{proposition}
 Let $(A, \mu, \varepsilon)$ be an associative color algebra and $\alpha : A\rightarrow A$ an even linear map such that $(A, \mu, \varepsilon, \alpha)$
be a Hom-associative color algebra. Then, for any fixed element $\xi\in A$, the quadruple $(A, \mu_\xi, \varepsilon, \alpha)$ is a Hom-associative
color algebra with
$$\mu_\xi(x, y)=x\xi y,$$
for any $x, y \in \mathcal{H}(A)$.
\end{proposition}
\begin{proof}
For any $x, y, z \in \mathcal{H}(A)$,
 \begin{eqnarray}
  as_{\mu_\xi}(x, y, z)&=&\mu_\xi(\mu_\xi(x, y), \alpha(z))-\mu_\xi(\alpha(x), \mu_\xi(y, z))\nonumber\\
&=&(x\xi y)\xi\alpha(z)-\alpha(x)\xi(y\xi z)\nonumber\\
&=&x(\xi y\xi)\alpha(z)-\alpha(x)(\xi y\xi) z\nonumber\\
&=&(x\xi y\xi)\alpha(z)-(x\xi y\xi)\alpha(z)\nonumber\\
&=&0.
 \end{eqnarray}
This gives the conclusion.
\end{proof}

% Hom-associative algebra of associative type. 

\begin{definition}
 A Hom-Lie color  algebra is a color Hom-algebra $(A, [\cdot, \cdot], \varepsilon, \alpha)$  such that
\begin{eqnarray}
 [x, y]=-\varepsilon(x, y)[y, x],\label{ss}\qquad\qquad\qquad\qquad\\
\varepsilon(z, x)[\alpha(x), [y, z]]+\varepsilon(x, y)[\alpha(y), [z, x]]+\varepsilon(y, z)[\alpha(z), [x, y]]=0,\label{chli}
\end{eqnarray}
for any $x, y, z\in\mathcal{H}(A)$.
\end{definition}
\begin{example}
It is clear that Lie color algebras are examples of Hom-Lie color algebras
by setting $\alpha = id$ . If, in addition, $\varepsilon(x, y)=1$ or $\varepsilon(x, y)=(-1)^{|x||y|}$, then the Hom-Lie color
algebra is nothing but a classical Lie algebra or Lie superalgebra. Hom-Lie algebras and
Hom-Lie superalgebras are also obtained when $\varepsilon(x, y)=1$ and $\varepsilon(x, y)=(-1)^{|x||y|}$ respectively.
 See \cite{LY} for other examples.
\end{example}
\begin{example}
 See \cite{LY} for Hom-Lie color $sl(2, \mathbb{K})$, Heisenberg Hom-Lie color algebra and Hom-Lie color algebra of Witt type.
\end{example}

\begin{definition} 
1) A  Rota-Baxter Hom-associative color algebra of weight $\lambda\in{\bf K}$ is a Hom-associative color 
algebra $(A, \cdot, \varepsilon, \alpha)$ together with an even linear map $R : A\rightarrow A$ that satisfies the identities
\begin{eqnarray}
R\circ\alpha&=&\alpha\circ R, \label{rb11}\\
R(x)\cdot R(y) &=& R\Big(R(x)\cdot y + x\cdot R(y) +\lambda x\cdot y\Big),\label{rb12}
\end{eqnarray}
for all $x, y\in\mathcal{H}(L)$.\\
2) A  Rota-Baxter Hom-Lie color algebra of weight $\lambda\in{\bf K}$ is a Hom-Lie color 
algebra $(L, [-, -], \varepsilon, \alpha)$ together with an even linear map $R : L\rightarrow L$ that satisfies the identities
\begin{eqnarray}
R\circ\alpha&=&\alpha\circ R, \label{rb21}\\
{[R(x), R(y)]} &=& R\Big([R(x), y] + [x, R(y)] +\lambda [x, y]\Big),\label{rb22}
\end{eqnarray}
for all $x, y\in\mathcal{H}(L)$.
\end{definition}

\begin{example}
Let $G=\{-1, +1\}$ be an abelian group and $A=A_{(-1)}\oplus A_{(1)}=<e_2>\oplus<e_1>$ a $G$-graded two dimensional vector space. The quintuple 
$(A, \cdot, \varepsilon, \alpha, R)$ is a Rota-Baxter Hom-associative color algebra of weight $\lambda$ with 
\begin{itemize}
 \item the multiplication, $e_1\cdot e_1=-e_1,\quad e_1\cdot e_2=e_2,\quad e_2\cdot e_1=e_2,\quad e_2\cdot e_2=e_1$,
\item the bicharacter, $\varepsilon(i, j)=(-1)^{(i-1)(j-1)/4}$,
\item the even linear map $\alpha : A\rightarrow A$ defined by : $\alpha(e_1)=e_1,\quad \alpha(e_2)=-e_2$,
\item the Rota-Baxter operator $R : A\rightarrow A$ given by : $R(e_1)=-\lambda e_1, R(e_2)=-\lambda e_2$.
\end{itemize}
\end{example}

\begin{definition}
For any integer $k$, we call \\
1)  an $\alpha^k$-averaging operator over a Hom-associative color algebra  $(A, \mu, \varepsilon, \alpha)$, an even linear map $\beta : A\rightarrow A$ such
that $\alpha\circ\beta=\beta\circ\alpha$ and
\begin{eqnarray}
\alpha\circ\beta&=&\beta\circ\alpha, \label{avo11}\\
 \beta(\mu(\beta(x), \alpha^k(y))&=&\mu(\beta(x), \beta(y))=\beta(\mu(\alpha^k(x), \beta(y))), \label{avo12}
\end{eqnarray}
for all $x, y\in\mathcal{H}(A)$.\\
2) an $\alpha^k$-averaging operator over a Hom-Lie color algebra  $(A, {-, -}, \varepsilon, \alpha)$, an even linear map 
$\beta : A\rightarrow A$ such that
\begin{eqnarray}
\alpha\circ\beta&=&\beta\circ\alpha,\label{avo21}\\
{[\beta(x), \beta(y)]}&=& \beta([\beta(x), \alpha^k(y)]),\label{avo22}
\end{eqnarray}
for all $x, y\in\mathcal{H}(A)$.
 \end{definition}
For example, any even $\alpha$-differential operator $d : A\rightarrow A$ (i.e. an $\alpha$-derivation $d$ such that $d^2=0$) over a Hom-associative
 color algebra  is an  $\alpha$-averaging operator.

\begin{definition}
For any  integer $k$, we call \\
1) an element of $\alpha^k$-centroid  of a Hom-associative color algebra $(L, \cdot, \varepsilon, \alpha)$,
 an even linear map $\beta : A\rightarrow A$ such that 
\begin{eqnarray}
\beta\circ\alpha&=&\alpha\circ\beta, \label{cent11}\\
 \beta(x\cdot y)&=&\beta(x)\cdot \alpha^k(y)=\alpha^k(x)\cdot \beta(y), \label{cent12}
\end{eqnarray}
for all $x, y\in\mathcal{H}(A)$.\\
2) an element of $\alpha^k$-centroid  of a Hom-Lie color algebra $(L, [-, -], \varepsilon, \alpha)$,
  an even linear map $\beta : L\rightarrow L$ such that 
\begin{eqnarray}
\beta\circ\alpha&=&\alpha\circ\beta, \label{cent21}\\
 \beta([x, y])&=&[\beta(x), \alpha^k(y)], \label{cent22}
\end{eqnarray}
for all $x, y\in\mathcal{H}(A)$.
\end{definition}
Observe that  $\beta([x, y]=[\alpha^k(x), \beta(y)]$ thanks to the $\varepsilon$-skew-symmetry.

\begin{definition}
 1) A Nijenhuis operator over a Hom-associative color algebra  $(A, \mu, \varepsilon, \alpha)$ is an even linear map $N : A\rightarrow A$ such that
\begin{eqnarray}
\alpha\circ N&=&N\circ\alpha,\label{nij11}\\
 \mu(N(x), N(y))&=&N\Big(\mu(N(x), y)+\mu(x, N(y))-N(\mu(x, y))\Big), \label{nij12}
\end{eqnarray}
for all $x, y\in\mathcal{H}(A)$.\\
2) A Nijenhuis operator over a Hom-Lie color algebra  $(A, \mu, \varepsilon, \alpha)$ is an even linear map $N : A\rightarrow A$ such that 
\begin{eqnarray}
\alpha\circ N&=&N\circ\alpha,\label{nij21}\\
 {[N(x), N(y)]}&=&N\Big([N(x), y]+[x, N(y)]-N([x, y])\Big), \label{nij22}
\end{eqnarray}
for all $x, y\in\mathcal{H}(A)$.
 \end{definition}

\section{Hom-Poisson color algebras}
In this section, we present some constructions of Hom-Poisson color algebras. In the most proof, we only prove the 
compatibility condition and left Hom-associativity and Hom-Jacobi identity.
\begin{definition}
  A Hom-Poisson color  algebra consists of a $G$-graded vector space $A$, a multiplication $\mu : A\times A\rightarrow A$,
 an even bilinear bracket $\{\cdot, \cdot\} : A\times A\rightarrow A$ and an even linear map $\alpha : A\rightarrow A$ such that :
\begin{enumerate}
 \item[1)] $(A, \mu, \varepsilon, \alpha)$ is a Hom-associative color algebra,
\item [2)]$(A, \{\cdot, \cdot\}, \varepsilon, \alpha)$ is a Hom-Lie color  algebra,
\item[3)]the Hom-Leibniz color identity
\begin{eqnarray}
 \{\alpha(x), \mu(y, z)\}=\mu(\{x, y\}, \alpha(z))+\varepsilon(x,y)\mu(\alpha(y), \{x, z\}),\label{cca}\nonumber
\end{eqnarray}
is satisfied for any $x, y, z\in \mathcal{H}(A)$.
\end{enumerate}
A Hom-Poisson color algebra $(A, \mu, \{\cdot, \cdot\}, \varepsilon, \alpha)$ in which $\mu$ is $\varepsilon$-commutative
is said to be a commutative Hom-Poisson color algebra.
 \end{definition}

\begin{example}\label{exp}
Let $A=A_{(0)}\oplus A_{(1)}=<e_1, e_2>\oplus<e_3>$ be  a  three-dimensional  graded vector space and $\cdot : A\times A\rightarrow A$ and 
$[-,-] : A\times A\rightarrow A$ the multiplications defined by 
\begin{center}
\begin{tabular}{lll}
 $e_1\cdot e_1= e_1$,& $e_1\cdot e_2= e_2$& $e_1\cdot e_3=a e_3$,     \\    
$e_2\cdot e_1= e_2$,&    $e_2\cdot e_1=\frac{1}{a} e_2$,& $e_2\cdot e_3=e_3$,\\
$e_3\cdot e_1=a e_3,$&     $ e_3\cdot e_2=0$& $e_3\cdot e_3=0$,\\
$[e_2, e_3]=e_3$ &   $[e_1, e_2]=0$   & $[e_1, e_3]=0$.
\end{tabular}
\end{center}
Then, the quintuple $(P, \cdot,  [-, -], \varepsilon, \alpha)$ is a Hom-Poisson color algebra with 
$$\alpha(e_1)=e_1, \quad \alpha(e_2)=e_2, \quad \alpha(e_3)=ae_3,$$
 and any bicharacter $\varepsilon$.
\end{example}

% \begin{example}
%  fdd
% \end{example}

\begin{theorem}
Let $(P, \ast, [-, -], \varepsilon, \alpha)$ be a Hom-Poisson color algebra and 
 $\sigma : G\times G\rightarrow \mathbb{K}^*$ be a {\it symmetric multiplier} on $G$ i.e.
\begin{enumerate}
 \item [i)] $\sigma(x, y)=\sigma(y, x), \forall x, y\in G$
\item [ii)] $\sigma(x, y)\sigma(z, x+y)$ is invariant under cyclic permutation of $x, y, z\in G$.
\end{enumerate}
Then, $(P, \ast^\sigma, [-, -]^\sigma, \varepsilon, \alpha)$ is also a Hom-Poisson color algebra with
$$x\ast^\sigma y=\sigma(x, y)x\ast y\quad\mbox{and}\quad[x, y]^\sigma=\sigma(x, y)[x, y],$$
for any $x, y\in \mathcal{H}(P)$.
\end{theorem}
\begin{proof}
For any homogeneous elements $x, y, z\in P$,
 \begin{eqnarray}
  \{\alpha(x), y\ast z\}
&=&\{\alpha(x),  \sigma(y, z)y z\}\nonumber\\
&=&\sigma(y, z)\sigma(x, y+z)[\alpha(x), yz]\\
&=&\sigma(y, z)\sigma(x, y+z)[x, y]\alpha(z)+\sigma(y, z)\sigma(x, y+z)\varepsilon(x, y)\alpha(y)\cdot[x, z]\nonumber\\
&=&\sigma(x, y)\sigma(z, x+y)[x, y]\alpha(z)+\sigma(z, x)\sigma(y, z+x)\varepsilon(x, y)\alpha(y)\cdot[x, z]\nonumber\\
&=&\sigma(z, x+y)\{x, y\}\alpha(z)+\sigma(y, x+z)\varepsilon(x, y)\alpha(y)\cdot\{ x, z\}\nonumber\\
&=&\{x, y\}\ast\alpha(z)+\varepsilon(x, y)\alpha(y)\ast\{ x, z\}.\nonumber\\
 \end{eqnarray}
% By symmetry and cyclic permutation, we have
% \begin{eqnarray}
%  \{\alpha(x), y\ast z\}=\sigma(x, y)\sigma(x+y, z)[x, y]\alpha(z)+
% \end{eqnarray}
% 
% $\{x, y\}\ast\alpha(z)+\varepsilon()$
This ends the proof.
\end{proof}

The following Theorem is proved as the previous one.
% We need the following definition for next theorem.
\begin{theorem}
Let $(P, \ast, [-, -], \varepsilon, \alpha)$ be a Hom-Poisson color algebra and 
 $\delta : G\times G\rightarrow \mathbb{K}^*$ be the bicharacter associated with the   {multiplier} $\sigma$ on $G$ 
Then, $(P, \ast^\sigma, [-, -]^\sigma, \varepsilon\delta, \alpha)$ is also a Hom-Poisson color algebra with
$$x\ast^\sigma y=\sigma(x, y)x\ast y,\quad[x, y]^\sigma=\sigma(x, y)[x, y]\quad\mbox{and}\quad \varepsilon\delta(x, y)
=\varepsilon(x, y)\sigma(x, y)\sigma(y, x)^{-1},$$
for any $x, y\in \mathcal{H}(P)$.\\
Moreover, an endomorphism of $(P, \cdot,  [-, -], \varepsilon, \alpha)$ is also an endomorphism of $(P, \ast^\sigma,  [x, y]^\sigma, \varepsilon, \alpha)$.
\end{theorem}

\begin{theorem}
 Let $(P', \cdot', [-, -]', \varepsilon, \alpha')$ be a Hom-Poisson color algebra and $P$ a graded vector space with a $\varepsilon$-skew-symmetric 
even bilinear bracket and an even linear map $\alpha$. Let $f : L\rightarrow L'$ be an even bijective linear map such that $f\circ\alpha=\alpha'\circ f$, 
$$f(x\cdot y)=f(x)\cdot' f(y)\quad\mbox{and}\quad f([x, y])=[f(x), f(y)]', \forall x, y\in \mathcal{H}(P).$$
Then $(P, \cdot, [-, -], \varepsilon, \alpha)$ is a Hom-Poisson color algebra.
\end{theorem}
\begin{proof}
 
 For any $x, y, z\in \mathcal{H}$,
\begin{eqnarray}
 [\alpha(x), y\cdot z]
&=&f^{-1}[f(\alpha(x)), f(y\cdot z)]'\nonumber\\
&=&f^{-1}[f(\alpha(x)), f\Big(f^{-1}\Big(f(y)\cdot' f(z)\Big)\Big)]'\nonumber\\
&=&f^{-1}[\alpha'(f(x)), f(y)\cdot' f(z)]'\nonumber\\
&=&f^{-1}\Big([f(x), f(y)]'\cdot'\alpha'(f(z))+\varepsilon(x, y) \alpha'(f(y))\cdot' [f(x)\cdot' f(z)]'\Big)\nonumber\\
&=&f^{-1}\Big(f\Big(f^{-1}[f(x), f(y)]'\Big)\cdot'\alpha'(f(z))\Big)
+\varepsilon(x, y) f^{-1}\Big(\alpha'(f(y))\cdot' f\Big(f^{-1}[f(x)\cdot' f(z)]'\Big)\Big)\nonumber\\
&=&f^{-1}\Big(f([x, y])\cdot'f(\alpha(z))\Big)
+\varepsilon(x, y) f^{-1}\Big(f(\alpha(y))\cdot' f([x\cdot z])\Big)\nonumber\\
&=&[x, y]\cdot\alpha(z)+\varepsilon(x, y) \alpha(y)\cdot [x\cdot z]\nonumber.
\end{eqnarray}
This gives the conclusion.
\end{proof}

\begin{definition}
Let $(P, \cdot, [-, -], \varepsilon, \alpha)$ be a Hom-Poisson color algebra. An even linear map $\beta : P\rightarrow P$ is said to be 
 \begin{enumerate}
  \item [1)] an element of $\alpha^k$-centroid of $P$ if, (\ref{cent11}),  (\ref{cent12}) and (\ref{cent22})  hold.
 \item [2)] an $\alpha^k$-averaging operator of $P$ if, (\ref{avo11}), (\ref{avo12}) and (\ref{avo22}) hold.
 \item [3)] a Rota-Baxter operator over $P$ if, (\ref{rb11}), (\ref{rb12}) and (\ref{rb22}) hold.
 \item [4)] a Nijenhuis operator over $P$ if, (\ref{nij11}), (\ref{nij12}) and (\ref{nij22}) hold.
 \end{enumerate}
\end{definition}

\begin{example}
The even linear map $R : P\rightarrow P$ defined, on the Hom-Poisson color algebra of Example \ref{exp}, by
$$R(e_1)=-\lambda e_1, \quad R(e_2)=-\lambda e_2, \quad R(e_3)=-\lambda e_3,$$
is a Rota-Baxter operator of weight $\lambda$ on $P$.
\end{example}

\begin{example}
 If $(A, \mu, \varepsilon, \alpha, R)$ is a Rota-Baxter Hom-associative color algebra, then 
$$P(A)=(A, \mu, \{\cdot, \cdot\}=\mu-\varepsilon(\cdot, \cdot)\mu^{op}, \varepsilon, \alpha),$$
 is a Rota-Baxter Hom-Poisson color algebra.
\end{example}

\begin{example}
 If $(A, \mu, \varepsilon, \alpha)$ is a Hom-associative color algebra and $N : A\rightarrow A$ a Nijenhuis operator on $A$, then is also a Nijenhuis
operator on 
$$P(A)=(A, \mu, \{\cdot, \cdot\}=\mu-\varepsilon(\cdot, \cdot)\mu^{op}, \varepsilon, \alpha).$$
\end{example}

\begin{theorem}
 Let $(P, \cdot, [-, -], \varepsilon, \alpha)$ be a Hom-Poisson color algebra and $\beta : P\rightarrow P$ be
 an element of $\alpha^0$-centroid of $P$. Let'us define a new multiplications $\ast: P\times P\rightarrow P$ and $\{-, -\}: P\times P\rightarrow P$ by
\begin{eqnarray}
 x\ast y=x\cdot y\quad\mbox{and}\quad \{x, y\}=[\beta(x), y], \forall x, y\in \mathcal{H}(P).
\end{eqnarray}
Then $(P, \ast, \{-, -\}, \varepsilon, \alpha)$ is also a Hom-Poisson color algebra.\\
Moreover, $\beta$ is a morphism of Hom-Poisson color algebra of $(P, \ast,  \{x, y\}, \varepsilon, \alpha)$
 onto $(P, \cdot,  [-, -], \varepsilon, \alpha)$.
\end{theorem}
\begin{proof}
 d
\end{proof}

\begin{theorem}
 Let $(P, \cdot, [-, -], \varepsilon, \alpha)$ be Hom-Poisson color algebra and $\beta : P\rightarrow P$ be an  $\alpha^0$-averaging 
operator. Then the two new products
\begin{eqnarray}
 x\ast y=\beta(x)\cdot \beta(y)\quad\mbox{and}\quad \{x, y\}=[\beta(x), \beta(y)], \forall x, y\in \mathcal{H}(P)
\end{eqnarray}
makes $(P, \ast, \{-, -\}, \varepsilon, \alpha)$ a Hom-Poisson color algebra.\\
% Moreover, $\beta$ is a morphism of Hom-Poisson color algebra of $(P, \ast,  \{x, y\}, \varepsilon, \alpha)$
%  onto $(P, \cdot,  [-, -], \varepsilon, \alpha)$.
 \end{theorem}
\begin{proof}

Let us prove the Hom-Leibniz color identity; for any $x, y, z\in\mathcal{H}(P)$,
 \begin{eqnarray}
  \{\alpha(x), x\ast y\}
&=&[\beta(x), \beta(y\ast z)]\nonumber\\
&=&[\beta\alpha(x), \beta(\beta(y)\beta(z))]\nonumber\\
&=&[\alpha\beta(x), \beta^2(y)\beta(z)]\nonumber\\
&=&[\beta(x), \beta^2(y)]\alpha\beta(z)+\varepsilon(x, y)\alpha\beta^2(y)[\beta(x), \beta(z)]\nonumber\\
&=&\beta[\beta(x), \beta(y)]\beta\alpha(z)+\varepsilon(x, y)\alpha\beta^2(y)[\beta(x), \beta(z)]\nonumber.
 \end{eqnarray}
The even linear map being an $\alpha^0$-averaging operator, it comes
 \begin{eqnarray}
  \{\alpha(x), x\ast y\}
&=&\beta[\beta(x), \beta(y)]\beta\alpha(z)+\varepsilon(x, y)\beta\alpha\beta(y)\beta[x, \beta(z)]\nonumber\\
&=&\beta[\beta(x), \beta(y)]\beta\alpha(z)+\varepsilon(x, y)\beta\Big(\beta\alpha(y)[\beta(x), \beta(z)]\Big)\nonumber\\
&=&\beta[\beta(x), \beta(y)]\beta\alpha(z)+\varepsilon(x, y)\beta\alpha(y)\beta[\beta(x), \beta(z)]\nonumber\\
&=&\{x, y\}\ast\alpha(z)+\varepsilon(x, y)\alpha(y)\ast\{x, z\}\nonumber.
 \end{eqnarray}
This finishes the proof.
\end{proof}

\begin{theorem}
 Let $(P, \cdot, [-, -], \varepsilon)$ be Poisson color algebra and $\beta : P\rightarrow P$ an  $\alpha^0$-averaging 
operator. Then the new products
\begin{eqnarray}
 x\ast y=\beta(x)\cdot y\quad\mbox{and}\quad \{x, y\}=[\beta(x), y], \forall x, y\in \mathcal{H}(P)
\end{eqnarray}
makes $(P, \ast, \{-, -\}, \varepsilon, \beta)$ into a Hom-Poisson color algebra.
 \end{theorem}

\begin{theorem}
 Let $(P, \cdot, [-, -], \varepsilon, \alpha)$ be Hom-Poisson color algebra and $\beta : P\rightarrow P$ be a bijective  $\alpha^k$-averaging 
operator. Then the new products
\begin{eqnarray}
 x\ast y=\beta(x)\cdot \alpha^k(y)\quad\mbox{and}\quad \{x, y\}=[\beta(x), \alpha^k(y)], \forall x, y\in \mathcal{H}(P)
\end{eqnarray}
makes $(P, \ast, \{-, -\}, \varepsilon, \alpha)$ a Hom-Poisson color algebra.\\
Moreover, $\beta$ is a morphism of Hom-Poisson color algebra of $(P, \ast,  \{x, y\}, \varepsilon, \alpha)$
 onto $(P, \cdot,  [-, -], \varepsilon, \alpha)$.
 \end{theorem}
\begin{proof}
Let us prove the Hom-Leibniz color identity; for any $x, y, z\in\mathcal{H}(P)$,
 \begin{eqnarray}
  \beta(\{\alpha(x), y\ast z\})
&=&[\beta\alpha(x), \alpha^k(y\ast z)]=[\beta\alpha(x), \alpha^k(\beta(y)\alpha^k(z))]\nonumber\\
&=&[\beta\alpha(x), \beta(\beta(y)\alpha^k(z))]=[\alpha\beta(x), \beta(y)\beta(z))]\nonumber\\
&=&[\beta(x), \beta(y)]\alpha\beta(z) +\varepsilon(x, y)\alpha\beta(y)\cdot[\beta(x), \beta(z))]\nonumber\\
&=&\beta[\beta(x), \alpha^k(y)]\alpha\beta(z) +\varepsilon(x, y)\alpha\beta(y)\cdot\beta[\beta(x), \alpha^k(z))]\nonumber\\
&=&\beta\Big(\beta[\beta(x), \alpha^k(y)]\alpha^{k+1}(z) +\varepsilon(x, y)\beta(y)\cdot\alpha^{k+1}[\beta(x), \alpha^k(z))]\Big)\nonumber\\
&=&\beta\Big(\beta\{x, y\}\alpha^{k+1}(z) +\varepsilon(x, y)\beta(y)\cdot\alpha^{k+1}\{x, z\}\Big)\nonumber\\
&=&\beta\Big(\{x, y\}\ast\alpha(z) +\varepsilon(x, y)\alpha(y)\ast\{x, z\}\Big)\nonumber.
 \end{eqnarray}
The associativity and Hom-Jacobi identity are proved in the same way.
\end{proof}

\begin{theorem}
 Let $(P, \cdot,  [-, -], \varepsilon, \alpha)$ be a Hom-Poisson color algebra and $N: P\rightarrow P$ be a Nijenhuis operator.
 Then the new multiplications $\ast_N : P \rightarrow P$ and  $[-, -]_N : P \rightarrow P$ given by 
\begin{eqnarray}
 x\cdot_N y=N(x)\cdot y+ x\cdot N(y)-N(x\cdot y) \quad\mbox{and}\quad   [x, y]_N=[N(x), y]+ [x, N(y)]-N([x, y]),
\end{eqnarray}
makes $P$ into a Hom-Poisson color algebra.\\
Moreover, $N$ is a morphism of Hom-Poisson color algebra of $(P, \cdot_N,  [x, y]_N, \varepsilon, \alpha)$
 onto $(P, \cdot,  [-, -], \varepsilon, \alpha)$.
\end{theorem}
\begin{proof}
On the one hand, we have for any $x, y, z\in \mathcal{H}$,
\begin{eqnarray}
 \{\alpha(x), y\cdot_N z\}
&=&\{\alpha(x), N(y)z+yN(z)-N(yz)\}\nonumber\\
&=&[N(\alpha(x)), N(y)z+yN(z)-N(yz)]+[\alpha(x), N(N(y)z+yN(z)-N(yz))]\nonumber\\
&&-N[\alpha(x),  N(y)z+yN(z)-N(yz)]\nonumber\\
&=&[N(\alpha(x)),  N(y)z] +[N(\alpha(x)), yN(z)]-N\Big([N(\alpha(x)), yz]+[\alpha(x), N(yz)]\nonumber\\
&&-N[\alpha(x), yz]\Big)+[\alpha(x), N(y)N(z)]-N[\alpha(x),  N(y)z]-N[\alpha(x), yN(z)]\nonumber\\
&&+N[\alpha(x), N(yz)]\nonumber.
\end{eqnarray}
Using Hom-Leibniz color identity,
\begin{eqnarray}
 \{\alpha(x), y\cdot_N z\}
&=&[N(x),  N(y)]\alpha(z)+\varepsilon(x, y)\alpha(N(y))[N(x), z] +[N(x), y]\alpha(N(z))\nonumber\\
&&+\varepsilon(x, y)\alpha(y)[N(x), N(z)]-N([N(x), y]\alpha(z))-\varepsilon(x, y)N(\alpha(y)[N(x), z])\nonumber\\
&&-N[\alpha(x), N(yz)]-N^2([x, y]\alpha(z))-\varepsilon(x, y) N^2(\alpha(y)[x, z])+[x, N(y)]\alpha(N(z))\nonumber\\
&&+\varepsilon(x, y)\alpha(N(y))[x, N(z)]-N([x,  N(y)]\alpha(z))-\varepsilon(x, y)N(\alpha(N(y))[x,  z])\nonumber\\
&&-N([x, y]\alpha(N(z))-\varepsilon(x, y)N(\alpha(y)[x, N(z)])+N[\alpha(x), N(yz)]\nonumber.
\end{eqnarray}
On the other hand,
\begin{eqnarray}
 &&\qquad\{x, y\}\cdot_N\alpha(z)+\varepsilon(x, y)\alpha(y)\cdot_N\{x, z\}=\nonumber\\
&&=([N(x), y]+[x, N(y)]-N[x, y])\cdot_N \alpha(z)\nonumber\\
&&\qquad+\alpha(x)\cdot_N([N(y), z]+[y, N(z)]-N[y, z])\nonumber\\
&&=N\Big([N(x), y]+[x, N(y)]-N[x, y]\Big)\alpha(z)+\Big([N(x), y]+[x, N(y)]-N[x, y]\Big)\alpha(N(z))\nonumber\\
&&\qquad\qquad-N\Big(([N(x), y]+[x, N(y)]-N[x, y])\alpha(z)\Big)\nonumber\\
&&\qquad+\varepsilon(x, y)N(\alpha(y))\Big([N(x), z]+[x, N(z)]-N[x, z]\Big)\nonumber\\
&&\qquad\qquad+\varepsilon(x, y)\alpha(y)N\Big([N(x), z]+[x, N(z)]-N[x, z]\Big)\nonumber\\
&&\qquad-\varepsilon(x, y)N\Big(\alpha(y)([N(x), z]+[x, N(z)]-N[x, z])\Big)\nonumber.
\end{eqnarray}
By Nijenhuis identity, 
\begin{eqnarray}
 &&\qquad\{x, y\}\cdot_N\alpha(z)+\varepsilon(x, y)\alpha(y)\cdot_N\{x, z\}=\nonumber\\
&&=[N(x), N(y)]\alpha(z)+[N(x), y]\alpha(N(z))\nonumber\\
&&\qquad+[x, N(y)]\alpha(N(z))-N\Big(N[x, y])\alpha(z)+[x, y]N(\alpha(z)-N([x, y]\alpha(z))\Big)\nonumber\\
&&\qquad-N([N(x), y]\alpha(z))-N([x, N(y)]\alpha(z))+N(N[x, y])\alpha(z))\nonumber\\
&&\qquad+\varepsilon(x, y)N(\alpha(y))[N(x), z]+\varepsilon(x, y)N(\alpha(y))[x, N(z)]-\varepsilon(x, y)N(\alpha(y))N[x, z])\nonumber\\
&&\qquad\qquad+\varepsilon(x, y)\alpha(y)[N(x), N(z)]\nonumber\\
&&\qquad-\varepsilon(x, y)N(\alpha(y)[N(x), z])-\varepsilon(x, y)N(\alpha(y)[x, N(z)])+\varepsilon(x, y)N(\alpha(y)N[x, z])\nonumber\\
&&=[N(x), N(y)]\alpha(z)+[N(x), y]\alpha(N(z))\nonumber\\
&&\qquad+[x, N(y)]\alpha(N(z))-N([x, y]N(\alpha(z))+N^2([x, y]\alpha(z))\nonumber\\
&&\qquad-N([N(x), y]\alpha(z))-N([x, N(y)]\alpha(z))\nonumber\\
&&\qquad+\varepsilon(x, y)N(\alpha(y))[N(x), z]+\varepsilon(x, y)N(\alpha(y))[x, N(z)]\nonumber\\
&&\qquad-\varepsilon(x, y)\Big(N(\alpha(y))[x, z]+\alpha(y)N[x, z]-N(\alpha(y)[x, z])\Big)\nonumber\\
&&\qquad\qquad+\varepsilon(x, y)\alpha(y)[N(x), N(z)]\nonumber\\
&&\qquad-\varepsilon(x, y)N(\alpha(y)[N(x), z])-\varepsilon(x, y)N(\alpha(y)[x, N(z)])+\varepsilon(x, y)N(\alpha(y)N[x, z])\nonumber.
\end{eqnarray}
 By comparing, we get the Hom-Leibniz color identity. This proves the Theorem.
\end{proof}

\begin{theorem}
Let $(P, \cdot,  [-, -], \varepsilon, \alpha)$ be a Hom-Poisson color algebra and $R : P\rightarrow P$ be  Rota-Baxter operator of weight
 $\lambda\in{\bf K}$ on $P$. Then $P$ is a Hom-Poisson color algebra with
\begin{eqnarray}
x\ast y = R(x)\cdot y + x\cdot R(y) +\lambda x\cdot y\quad\mbox{and}\quad \{x, y\} = [R(x), y] + [x, R(y)] +\lambda [x, y],
\end{eqnarray}
for all $x, y\in\mathcal{H}(P)$.\\
Moreover, $R$ is a morphism of Hom-Poisson color algebra of $(P, \ast,  \{x, y\}, \varepsilon, \alpha)$
 onto $(P, \cdot,  [-, -], \varepsilon, \alpha)$.
\end{theorem}
\begin{proof}
 For any $x, y, z\in \mathcal{H}$,
\begin{eqnarray}
 \{\alpha(x), y\ast z\}
&=&\{\alpha(x), R(y)z+yR(z)+\lambda yz\}\nonumber\\
&=&[R\alpha(x), R(y)z+yR(z)+\lambda yz]+[\alpha(x), R(R(y)z+yR(z)+\lambda yz)]\nonumber\\
&&+\lambda[\alpha(x), R(y)z+yR(z)+\lambda yz]\nonumber\\
&=&[R\alpha(x), R(y)z]+[R\alpha(x), yR(z)]+\lambda[R\alpha(x), yz]+[\alpha(x), R(y)R(z)]\nonumber\\
&&+\lambda[\alpha(x), R(y)z]+\lambda[\alpha(x), yR(z)]+\lambda^2[\alpha(x), yz]\nonumber.
\end{eqnarray}
By Hom-Leibniz color identity,
\begin{eqnarray}
 \{\alpha(x), y\ast z\}
&=&[R(x), R(y)]\alpha(z)+\varepsilon(x, y)\alpha(R(y))[R(x), z]  +[R(x), y]\alpha(R(z))+\varepsilon(x, y)\alpha(y)\cdot[R(x), R(z)]\nonumber\\
&&+\lambda[R(x), y]\alpha(z)+\lambda\varepsilon(x, y)\alpha(y)\cdot[R(x), z]+[x, R(y)]\alpha(R(z))+\varepsilon(x, y)\alpha(R(y))[x, R(z)]\nonumber\\
&&+\lambda[x, R(y)]\alpha(z)+\varepsilon(x, y)\lambda\alpha R(y)[x, z]+\lambda[x, y]\alpha(R(z))
+\lambda\varepsilon(x, y)\alpha(y)[x, R(z)]\nonumber\\
&&+\lambda^2[x, y]\alpha(z)+\lambda^2\varepsilon(x, y)\alpha(y)[x, z]\nonumber.
\end{eqnarray}
By reorganizing the terms, we have
\begin{eqnarray}
 \{\alpha(x), y\ast z\}
&=&[R(x), R(y)]\alpha(z)+[R(x), y]\alpha(R(z))+\lambda[x, y]\alpha(R(z))+[x, R(y)]\alpha(R(z))\nonumber\\
&&+\lambda[R(x), y]\alpha(z)+\lambda[x, R(y)]\alpha(z)+\lambda^2[x, y]\alpha(z)\nonumber\\
&&+\varepsilon(x, y)\alpha(R(y))[R(x), z]+\varepsilon(x, y)\alpha(R(y))[x, R(z)]+\varepsilon(x, y)\lambda\alpha R(y)[x, z]\nonumber\\
&&+\varepsilon(x, y)\alpha(y)[R(x), R(z)]+\lambda\varepsilon(x, y)\alpha(y)[R(x), z]
+\lambda\varepsilon(x, y)\alpha(y)[x, R(z)]\nonumber\\
&&+\lambda^2\varepsilon(x, y)\alpha(y)[x, z]\nonumber\\
&=&R\Big([R(x), y]\alpha(z)+[x, R(y)]+\lambda[x, y]\Big)\alpha(z)\nonumber\\
&&+\Big([R(x), y]+[x, R(y)]+\lambda[x, y]\Big)\alpha(R(z))\nonumber\\
&&+\lambda\Big([R(x), y]+[x, R(y)]+\lambda[x, y]\Big)\alpha(z)\nonumber\\
&&+\varepsilon(x, y)\alpha(R(y))\Big([R(x), z]+[x, R(z)]+\lambda[x, z]\Big)+\varepsilon(x, y)\alpha(y)[R(x), R(z)]\nonumber\\
&&+\lambda\varepsilon(x, y)\alpha(y)\Big([R(x), z]+[x, R(z)]+\lambda[x, z]\Big).\nonumber\\
&=&\Big([R(x), y]+[x, R(y)]+\lambda[x, y]\Big)\ast\alpha(z)\nonumber\\
&&+\varepsilon(x, y)\alpha(y)\ast\Big([R(x), z]+[x, R(z)]+\lambda[x, z]\Big).\nonumber\\
&=&\{x, y\}\ast\alpha(z)+\varepsilon(x, y)\alpha(y)\ast\{x, z\}.\nonumber
\end{eqnarray}
The Hom-associativity and Hom-Jacobi color identity are proved in a similar way.
\end{proof}

The below theorem asserts that the tensor product of any commutative associative algebra and any $n$-Hom-Lie color algebra gives rise to an $n$-Hom-Lie
color algebra.
\begin{theorem}
Let $(P, \cdot, [-, -], \varepsilon, \alpha)$ be a Hom-Poisson color algebra over a field $\mathbb{K}$ and $\hat{\mathbb{K}}$
 an extension of $\mathbb{K}$ (bdsvcyd).  Then, the graded $\mathbb{K}$-vector space 
 $$\hat{\mathbb{K}}\otimes P=\sum_{g\in G}(\mathbb{K}\otimes P)_g=\sum_{g\in G}\mathbb{K}\otimes P_g$$
is an Hom-Poisson color algebra with :\\ 
the associative product
$$ (\xi\otimes x)\cdot'(\eta\otimes y):=\xi\eta\otimes (x\cdot y),$$
the bracket
$$[\xi\otimes x,  \eta\otimes y]'=\xi\eta\otimes [x, y],$$
the even linear map
$$\alpha'(\xi\otimes x):=\xi\otimes \alpha(x)$$
and the bicharacter
$$\varepsilon(\xi+x, \eta+y)=\varepsilon(x, y), \forall \xi, \eta\in\hat{\mathbb{K}} , \forall x, y\in \mathcal{H}(P).$$
\end{theorem}

\begin{theorem}
 Let $(A, \cdot, \varepsilon, \alpha_A)$  be a commutative Hom-associative
color algebra and $(P, \ast, \{-, -\}, \varepsilon, \alpha_L)$ be a Hom-Poisson color algebra. 
Then the tensor product $A\otimes P$ endowed with the even linear map
 $\alpha=\alpha_A\otimes \alpha_P : A\otimes P\rightarrow A\otimes P$ and the even
 bilinear maps $\ast : (A\otimes P)\times (A\otimes P)\rightarrow A\otimes P$ and
 $\{-, -\} : (A\otimes P)\times (A\otimes P)\rightarrow A\otimes P$ defined, for any $a, b\in \mathcal{H}(A)$, $x, y\in \mathcal{H}(P)$,  by 
\begin{eqnarray}
  \alpha(a\otimes x)&:=&\alpha_A(a)\otimes\alpha_P(x),\nonumber\\
(a\otimes x)(b\otimes y)&:=&\varepsilon(x, b)a\cdot b\otimes x\ast y,\nonumber\\
\{a\otimes x, b\otimes y\}&:=&\varepsilon(x, b)(a\cdot b)\otimes [x, y],\nonumber
\end{eqnarray}
is a Hom-Poisson color algebra.
\end{theorem}
\begin{proof}
 For any $x, y, z\in \mathcal{H}$,
\begin{eqnarray}
 \{\alpha(a\otimes x), (b\otimes y)\ast(c\otimes z)\}
&=&\varepsilon(y, c)\{\alpha(a)\otimes\alpha(x), bc\otimes yz\}\nonumber\\
&=&\varepsilon(y, c)\varepsilon(x, b+c)\alpha(a)(bc)\otimes[\alpha(x), yz]\nonumber\\
&=&\varepsilon(y, c)\varepsilon(x, b+c)\alpha(a)(bc)\otimes([x, y]\alpha(z)+\varepsilon(x, y)\alpha(y)[x, z]\nonumber\\
&=&\varepsilon(y, c)\varepsilon(x, b) \varepsilon(x, c)\alpha(a)(bc)\otimes[x, y]\alpha(z)+\nonumber\\
&&\varepsilon(y, c)\varepsilon(x, b+c)\varepsilon(x, y)(ab)\alpha(c)\otimes\alpha(y)[x, z]\nonumber\\
&=&\varepsilon(x, b)\varepsilon(x+y, c)(ab)\alpha(c)\otimes[x, y]\alpha(z)+\nonumber\\
&&\varepsilon(y, c)\varepsilon(x, b+c)\varepsilon(x, y)\varepsilon(a, b)(ba)\alpha(c)\otimes\alpha(y)[x, z]\nonumber\\
&=&\varepsilon(x, b)(ab\otimes[x, y])(\alpha(c)\otimes\alpha(z))+\nonumber\\
&&\varepsilon(y, c)\varepsilon(x, b)\varepsilon(x, y)\varepsilon(a, b)(\alpha(b)\otimes\alpha(y))(ac\otimes[x, z])\nonumber\\
&=&\{a\otimes x, b\otimes y\}\alpha(c\otimes z)+\varepsilon(a+x, b+y)\alpha(b\otimes y), \{a\otimes x, c\otimes z\}.\nonumber
\end{eqnarray}

\end{proof}
%%%%%%%%%%%
% {\bf Acknowlegments :} 

\end{document}